\documentclass{siamltex}
\usepackage{amstext,amsmath,amssymb}         
\usepackage{mathtext}
\usepackage{multicol}
\usepackage{pgfplots}
\usepackage{makecell}
\usepackage{hyperref}
\usepackage{float}
\usepackage{placeins}
\usepackage{multicol}
\usepackage[noend]{algorithmic}
\usepackage{pgfplots}
\usepackage[colorinlistoftodos,textwidth=4cm,shadow]{todonotes}
\usepackage{algorithmic,algorithm}

\usepackage{amsthm}
\newtheorem{remark}{Remark}

\usepackage{graphicx,subcaption}

\DeclareMathOperator*{\argmin}{arg\,min}

\newcommand{\OP}{\mathrm{P}}
\newcommand{\RQ}{\mathfrak{R}}
\newcommand{\Grad}{\mathrm{grad}}
\newcommand{\Hess}{\mathrm{Hess}}
\newcommand{\M}{\mathcal{M}}
\newcommand{\VEC}{\operatorname{vec}}
\newcommand{\F}{\mathbb{R}}
\newcommand{\SM}{\mathcal{N}}%\mathcal{M}_r^{nm-1}}

\newtheorem{prop}{Proposition}

\newcounter{Ivan}

\newcounter{Maxim}

\pdfoutput=1
\title{Jacobi-Davidson method on low-rank matrix manifolds\thanks{This study was supported by the Ministry of Education and Science of the Russian Federation (grant 14.756.31.0001), by RFBR
grants 16-31-60095-mol-a-dk, 16-31-00372-mol-a and by Skoltech NGP program.}}

\author{M.~V. Rakhuba\footnotemark[2] 
\and I.~V. Oseledets\footnotemark[2] \footnotemark[3]}
\begin{document}
\maketitle

\renewcommand{\thefootnote}{\fnsymbol{footnote}}
\footnotetext[2]{Skolkovo Institute of Science and Technology,
Skolkovo Innovation Center, Building 3,
143026 Moscow,
Russia (\href{mailto:rakhuba.m@gmail.com}{rakhuba.m@gmail.com})}
\footnotetext[3]{Institute of Numerical Mathematics,
Gubkina St. 8, 119333 Moscow, Russia (\href{mailto:i.oseledets@skoltech.ru}{i.oseledets@skoltech.ru})}
%\slugger{sisc}{xxxx}{xx}{x}{x--x}%slugger should be set to mms, siap, sicomp, sicon, sidma, sima, simax, sinum, siopt, sisc, or sirev

%\bibliographystyle{unsrt}

\begin{abstract}
In this work we generalize the Jacobi-Davidson method to the case when eigenvector can be reshaped into a low-rank matrix.
In this setting the proposed method inherits advantages of the original Jacobi-Davidson method, has lower complexity and requires less storage.
We also introduce low-rank version of the Rayleigh quotient iteration which naturally arises in the Jacobi-Davidson method.
\end{abstract}

%\begin{keywords}
%multidimensional convolution, tensor train, tensor decompositions, multilinear algebra, cross approximation, black box approximation
%\end{keywords}

%\begin{AMS}
%15A69, 15B05, 44A35, 65F99
%\end{AMS}

\section{Introduction}

This paper considers generalization of the Jacobi-Davidson (JD) method \cite{svh-jd-2000} for finding target eigenvalue $\lambda$ (extreme or closest to a given number) and the corresponding eigenvector $x$ of $N\times N$ matrix~$A$:
$$
Ax = \lambda x.
$$
We treat the specific case when $N=nm$ and the eigenvector $x$ reshaped into \mbox{$n\times m$} matrix is exactly or approximately of small rank $r$.
For example, consider a Laplacian operator discretized on tensor product grid; its reshaped eigenvectors are of rank $1$.
%Low-rank structure of eigenvectors can be observed in different applications, e.g. in quantum chemistry computations.
For $r\ll n,m$ our assumption allows to significantly reduce storage of the final solution, at the same time leading to algorithmic complications that we address in this paper.

%At the same time low-rank assumption leads to algorithmic complications that we address in this paper.

%If $r\ll n,m$, then 
% assumption allows to reduce complexity, but at the same time leads to complications due to nonlinearity of the fixed-rank matrix manifold. 

%Under the assumption that the operator $A$ can be approximated as
%$$
%	A \approx  \sum_{\alpha=1}^R F_\alpha \otimes G_\alpha,
%$$
%where $F_\alpha$ and $G_\alpha$ are sparse $n\times n$ and $m\times m$ matrices the complexity of the proposed algorithm scales (Sec. \ref{jd:sec:complexity}) as~$\mathcal{O}((n+m) r R)$.
%
%but at the same time leads to algorithmic complications that we address in this paper.

Similarly to the original JD method, we derive the low-rank Jacobi correction equation and propose low-rank version of subspace acceleration.
%The proposed approach takes the advantage of the original JD method and hence is efficient both for the case when low-rank Jacobi correction equation is solved  accurately or inexactly.
The proposed approach takes the advantage of the original JD method. Compared with the Rayleigh quotient iteration and the Davidson approach, the method is efficient for the cases when arising linear systems are solved both accurately and inexactly.
%The benefit that we observe on the proposed generalization is that it provides a general strategy for both inexact and exact solution of sol Jacobi-Davidson method is an efficient way . Takes advantage
 
The JD method is known to be a Riemannian Newton method on a unit sphere $\{x: \| x \| = 1\}$ with additional subspace acceleration \cite{absil-manopt-2009}.
We utilize this interpretation and derive a new method as an inexact Riemannian Newton method on the intersection of  the sphere and the fixed-rank manifold.
%Similarly to \cite{ksv-manprec-2016} we omit a part of the Hessian that corresponds to the curvature of the fixed-rank manifold to make method practical in numerics.
In derivation we assume that the matrix~$A$ is real and symmetric, however we test our approach on non-symmetric matrices as well.
Complexity of the proposed algorithm scales as $\mathcal{O}\left((n+m) r (R+r)\right)$ if $A$ can be approximated as
$$
	A \approx  \sum_{\alpha=1}^R F_\alpha \otimes G_\alpha,
$$
where $F_\alpha$ and $G_\alpha$ allow  fast matrix-vector multiplication, e.g. they are sparse.
%We also introduce manifold-projected version of the inverse iteration, which naturally arises in the derivation.

%In this work we additionally utilize the information that the solution belongs to the manifold of matrices with the fixed rank $r$.
%The manfold of low-rank matrices is an example of the simplest nonlinear manifold compared with more complicated manifold of lwo-rank tensors.
%The efficient generalization of the proposed approach to the tensor case is a seperate topic and is not covered in this work.

Our main contributions are:
\begin{itemize}
\item We generalize the Jacobi correction equation (Sec. \ref{jd:sec:jaccor}) and the subspace acceleration (Sec. \ref{jd:sec:sub_acc}) to the case of fixed-rank matrix manifolds.
\item We introduce low-rank version of the Rayleigh quotient iteration (Sec. \ref{jd:sec:ii}) which naturally arises in the JD method.
%\item We build a block Jacobi preconditioner (Sec. \ref{jd:sec:prec}) to solve linear systems arising in the Jacobi correction equation.
%\item We investigate numerical behaviour of the proposed algorithm on both hermitian and non-hermitian matrices (Sec. \ref{jd:sec:num}).
\end{itemize}

%Hereafter we will also use ``JD'' abbreviation of the Jacobi-Davidson method.

%We utilize this idea and additionally exploit the information that the solution belongs to the manifold of matrices with the fixed rank $r$.
%As a result complexity of our algorithm scales linearly in $n$.

%We also propose a way for fast evaluation of the block Jacobi preconditioner that we use to solve local linear systems.
%We utilize the Riemannian optimization approach and solve the \emph{Jacobi correction equation} on the intersection of tangent space of manifold of matrices with fixed rank $r$ and  

%The complexity of the proposed algorithm scales linearly in $n$: $\mathcal{O}(n r^3)$.

%The Jacobi-Davidson

%We generalize the Jacobi-Davidson method

%The Jacobi-Davidson method is a popular method for finding several eigenvalues and corresponding eigenva
%This work is devoted to
%The goal is to generalize the Jacobi-Davidon method to the TT case.
%The main idea is that we choose the correction on the tangent space of current solution vector.
%We will also investigate how to solve local systems efficiently (possibly replacing by a block-diagonal matrix or use block-Jacobi preconditioner).

\section{Rayleigh quotient minimization on sphere}
The first ingredient of the JD method is the Jacobi correction equation. 
The Jacobi correction equation can be derived as a Riemannian Newton method on the unit sphere \cite{absil-manopt-2009}, which will be useful for our purposes.
In this section we provide the derivation, and in Sec.~\ref{jd:sec:jaccor} it will be generalized to the low-rank case.

Given a symmetric matrix $A\in\F^{n\times n}$ the goal is to optimize
\begin{equation}\label{jd:eq:rq}
	\RQ (x) = x^\top A x,
\end{equation}
subject to $x\in S^{n-1}$, where $S^{n-1}$ is a unit sphere considered as an embedded submanifold of $\mathbb{R}^n$ with the pullback metric $g_x(\xi, \eta) = \xi^\top \eta$. %inherited from the inner product on $\mathbb{R}^n$.
The Riemannian optimization approach implies that we optimize $\RQ (x)$ on $S^{n-1}$, i.e. constraints are already accounted for in the search space.
One of the key concepts in the Riemannian optimization is a \emph{tangent space} which is in fact a linearization of the manifold at a given point.
%\begin{equation}\label{jd:eq:rq}
%	\RQ_A: S^{n-1} \to \mathbb{R}: x \to x^\top A x.
%\end{equation}
The orthogonal projection of $\xi$ on the tangent space $T_{x} S^{n-1}$ of $S^{n-1}$ at $x$ can be written as \cite{absil-manopt-2009}
\begin{equation}\label{jd:eq:proj}
	\OP_{T_x S^{n-1}} \xi = (I - x x^\top) \xi.
\end{equation}
The Riemannian gradient of \eqref{jd:eq:rq} is
\begin{equation}\label{jd:eq:grad}
	\grad \RQ (x) = \OP_{T_x S^{n-1}} \nabla \RQ (x) = (I - x x^\top) (2Ax),
\end{equation}
where $\nabla$ denotes the Euclidean gradient.
The Hessian $\Hess_x: T_{x} S^{n-1} \to T_{x} S^{n-1}$ can be obtained as \cite{amt-riemhess-2013}
\begin{equation}\label{jd:eq:hess_sphere}
\begin{split}
	\Hess_x\, \RQ (x) [\xi] =& \OP_{T_x S^{n-1}} \left( \mathrm{D}\, \left( \Grad\, \RQ(x)\right) [\xi] \right) = \\
		& 2\OP_{T_x S^{n-1}} \left( \mathrm{D}\left( \OP_{T_x S^{n-1}} Ax\right) [\xi] \right) = \\
		 &2 \OP_{T_x S^{n-1}} (A \xi + \dot{\OP}_{T_x S^{n-1}} A x), \quad \xi \in T_{x} S^{n-1},
\end{split}
\end{equation}
where $\mathrm{D}$ denotes the differential map (directional derivative) and
$$
	\dot{\OP}_{T_x S^{n-1}} A x \equiv \mathrm{D}(\OP_{T_x S^{n-1}})[\xi] A x  = - (x^\top Ax) \xi - (\xi^\top A x) x.
$$
Since $\OP_{T_x S^{n-1}} x = 0$ and $\OP_{T_x S^{n-1}} \xi = \xi$ we arrive at
\begin{equation}\label{jd:eq:hess}
	\Hess_x\, \RQ (x) [\xi] = 2 \OP_{T_x S^{n-1}} \left(A - (x^\top A x) I \right) \OP_{T_x S^{n-1}} \xi.
\end{equation}
%We note that the second term in the Hessian $-(x^\top A x) I$ is a curvature term arising due to the nonlinearity of a sphere. 
The $k$-th step of the Riemannian Newton methods looks as
\begin{equation}\label{jd:eq:newt}
	\Hess_{x_k}\, \RQ (x_k) [\xi_k] =  - \grad \RQ (x_k), \quad \xi_k \in T_{x_k} S^{n-1},
\end{equation}
with the \emph{retraction}
\begin{equation}\label{jd:eq:retr_sphere}
	x_{k+1} = \frac{x_k + \xi_k}{\|x_k + \xi_k\|},
\end{equation}
which returns $x_k + \xi_k$ back to the manifold $S^{n-1}$.
Using \eqref{jd:eq:proj}, \eqref{jd:eq:grad} and \eqref{jd:eq:hess} we can rewrite \eqref{jd:eq:newt} as
\begin{equation}\label{jd:eq:jac_cor}
	(I - x_k x_k^\top) \left(A - \RQ(x_k) I \right)(I - x_k x_k^\top) \xi_k = - r_k, \quad x_k^\top \xi_k = 0,
\end{equation}
where
$$
\RQ(x_k) = x_k^\top A x_k, \quad  r_k = (I - x_k x_k^\top) Ax_k = A x_k - \RQ(x_k) x_k.
$$
Equation \eqref{jd:eq:jac_cor} is called the \emph{Jacobi correction equation} \cite{svh-jd-2000}.
Note that without the projection $(I - x_k x_k^\top)$ we obtain the Davidson equation
$$
\left(A - \RQ(x_k) I \right) \xi_k = - r_k,
$$
which has solution $\xi_k = -x_k$ collinear to the current approximation $x_k$.
This is the reason for the Davidson equation to be solved inexactly.
The original Davidson algorithm~\cite{d-davidson-1975} replaces $A$ by its diagonal part $\mathrm{diag}(A)$.
By contrast, even if the Jacobi correction equation \eqref{jd:eq:jac_cor} is solved inexactly using Krylov iterative methods, its solution $\xi_k$ will be automatically orthogonal to $x_k$ which is beneficial for the computational stability.
Moreover, since the JD method has the Newton interpretation it boasts local superlinear convergence.
%Moreover, thanks to the Newton interpretation JD method has

The goal of this paper is to extend the Jacobi correction equation \eqref{jd:eq:jac_cor} and the second ingredient of the JD method --- \emph{subspace acceleration} --- to the case of low-rank manifolds.

\section{Jacobi correction equation on fixed-rank manifolds} \label{jd:sec:jaccor}

%Let us utilize the assumption that the required eigenvector $X$ is of rank $r$.
%This means that $X$ being reshaped into a vector columnwise $x=\mathrm{vec}(X)$ belongs to 
Let $x\in\F^{nm}$ be an eigenvector of $A$ and $X\in\F^{n\times m}$ be its matricization: $x = \VEC (X)$, where $\VEC$ denotes columnwise reshape of $n\times m$ matrix into $nm$ vector.
In this paper we make an assumption that matricisized eigenvector $X$ is approximately of rank $r$.
Therefore, for example, to approximate the smallest eigenvalue we solve the following optimization problem
\begin{equation}\label{jd:eq:optimization}
\begin{split}
	\text{minimize}\quad &\RQ(x) = x^\top A x \\
	\text{s.t. }\quad   &x \in S^{nm-1}\cap \M_r, %\equiv \{ \mathrm{vec}(X),\, X\in \mathbb{R}^{n\times m}: \text{rank}(X) = r \}
\end{split}
\end{equation}
where
$$
	\M_r = \{ \mathrm{vec}(X),\, X\in \mathbb{R}^{n\times m}: \text{rank}(X) = r \},
$$
which forms a smooth embedded submanifold of $\F^{nm}$ of dimension $(m+n)r - r^2$ \cite{lee-introman-2001}.
By analogy with the derivation of the Jacobi equation we additionally intersected the manifold  $\M_r$ with the sphere~$S^{nm-1}$.
As we will see from the following proposition $S^{nm-1}\cap \M_r$ forms a smooth embedded submanifold of $\F^{nm}$.
Hence, optimization problem \eqref{jd:eq:optimization} can be solved using Riemannian optimization techniques.  
%Let us formulate the key properties of the intersection.
%One can easily check that intersection
%$$
%	\SM = \M_r \cap S^{nm-1}
%$$
%also forms smooth submanifold of $\F^{nm}$ thanks to the transversality property.
%To derive Hessian of $\RQ (x)$ on $\SM$ we first need to derive the orthogonal projection onto the tangent space.
%The assumption we use is that eigenvector additionally belongs to $\M_r$.
%Therefore we need to derive the Riemannian Newton method for the submanifold which is intersection of $S^{nm-1}$ and $\M_r$.
%Let us find orthogonal projection on its tangent space.

%\subsection{Properties of $\SM$}
%We start with discussing the tangent space of $\M_r$ at $\VEC (X)$, where $X$ is given by its SVD decomposition $X= U S V^\top$ of rank $r$ can be parametrized as \cite{todo}
%\[
%\begin{split}
%	T_X \M_r = \{ \VEC (U \Sigma V^\top + U_\perp V^\top + U V^\top_\perp), \ U_\perp\in\F^{n\times r},\, V_\perp\in\F^{m\times r},\, \Sigma\in\F^{r\times r},   \\ U_\perp^\top U = 0, \ V_\perp^\top V = 0, \}.
%\end{split}
%\]
\begin{prop}
Let $\SM =  S^{nm-1}\cap \M_r$, then
\begin{enumerate}
	\item $\SM$ forms smooth embedded submanifold of $\F^{nm}$ of dimension $(n+m)r-r^2 - 1$
	\item The tangent space of $\SM$ at $\VEC (X) \in \SM$ with $X$ given by SVD: $X= USV^\top$, $U^\top U = I$, $V^\top V = I$, $S = \mathrm{diag}(\sigma_1,\dots,\sigma_r)$, $\sigma_1\geq \dots \geq \sigma_r >0$ can be parametrized as 
		\[
		\begin{split}
		T_X \SM = \{\VEC (U_\xi V^\top + U V^\top_\xi + U S_\xi V^\top): U_\xi \perp U,\,  V_\xi \perp V,\, \VEC (S_\xi) \perp \VEC (S) \}.
		\end{split}
		\] 
	\item The orthogonal projection $\OP_{T_X \SM}$ onto $T_X \SM$ can be written as
	\begin{equation}\label{jd:eq:oproj}
	\begin{aligned}
	\OP_{T_X \SM} &=  \OP_{T_X \M_r}  \OP_{T_X S^{nm-1}} = \OP_{T_X S^{nm-1}} \OP_{T_X \M_r}\\
	&=\OP_{T_X \M_r}  - \mathrm{vec}(X)\mathrm{vec}^\top(X),
	\end{aligned}
	\end{equation} 
	where $\OP_{T_X \M_r}$ is the orthogonal projection onto the tangent space of $\M_r$:
	$$
	\OP_{T_X \M_r} = VV^\top \otimes UU^\top + 
	VV^\top \otimes (I_n - UU^\top) +  (I_m - VV^\top) \otimes UU^\top.
	$$
\end{enumerate}
%Orthogonal projection $\OP_{T_X \SM}$ onto a tangent space $\SM = \M_r \cap S^{nm-1}$ can be written as
%\begin{equation}\label{jd:eq:oproj}
%\begin{split}
%	\OP_{T_X \SM} = \OP_{T_X \M_r}  - \mathrm{vec}(X)\mathrm{vec}^\top(X).
%% \big [ V \otimes U\, (I - \mathrm{vec}(S)\mathrm{vec}^\top(S))\, V^\top \otimes U^\top + \\
%%	 &(I - VV^\top) \otimes UU^\top  + UU^\top \otimes (I - VV^\top)) \big] \mathrm{vec} (Y).
%\end{split}
%\end{equation}
\end{prop}
\begin{proof}
%Let us first formulate several properties that will be used in the proof.
%$\M_r$ forms smooth embedded submanifold of $\F^{nm}$ of dimension $(n+m)r-r^2$ \cite{lee-introman-2001}.
% and its tangent space can be parametrized as
%$$
%T_X \SM = \{\VEC (U_\xi V^\top + U V^\top_\xi + U S_\xi V^\top): U_\xi \perp U,\  V_\xi \perp V \}.
%$$
%Hence, the orthogonal projection on the tangent space 
%$
%	\OP_{T_X \M_r}: \mathbb{R}^{n m} \to T_X \M_r
%$
%is
%$$
%\OP_{T_X \M_r} \VEC(Y) = \VEC \left((UU^\top) Y (VV^\top) +(I_n - UU^\top) Y (VV^\top) + (UU^\top) Y (I_m - VV^\top)\right).
%$$
%Using the well-known Kronecker product property $\text{vec}(AXB) = (B^\top \otimes A)\, \text{vec}(X)$ we obtain
%\[
%\begin{split}
%	\OP_{T_X \M_r} = VV^\top \otimes UU^\top + 
%	VV^\top \otimes (I_n - UU^\top) +  (I_m - VV^\top) \otimes UU^\top .
%\end{split}
%\]

The first property follows from the fact that $\M_r$ and $S^{nm-1}$ are transversal embedded submanifolds of $\F^{nm}$. 
	Indeed, one can easily verify that $$T_X \M_r + T_X S^{nm-1} =\F^{nm}.$$
	Hence, by the transversality property \cite{lee-introman-2001} $\SM$ forms a smooth embedded submanifold of $\F^{nm}$ of dimension 
	$$
	\dim(\M_r) + \dim(S^{nm-1}) - \dim(\F^{nm}) = (n+m-r)r -1
	$$

Let us prove the second property of the proposition. Vector $\xi \in T_X \M_r$ can be parametrized \cite{bart-riemcompl2d-2013} as
\begin{equation}\label{jd:eq:parametrization}
	\xi = \VEC (U_\xi V^\top + U V^\top_\xi + U S_\xi V^\top)
\end{equation}
with the gauge conditions
\begin{equation}\label{jd:eq:gauge_uv}
	U_\xi \perp U, \quad V_\xi \perp V.
\end{equation}
To obtain the parametrization of $\xi \in T_{X} S^{nm-1} \cap T_X \M_r$ we need to take into account that $\xi \in T_{X} S^{nm-1}$ and, hence, $\xi^\top x =0$ yielding the additional gauge condition
\begin{equation}\label{jd:eq:gauge_s}
\text{vec} (S_\xi) \perp \text{vec} (S).
\end{equation}

Let us prove the third property by showing that operators $\OP_{T_X \M_r}$ and $\OP_{T_X S^{nm-1}}$ commute and, hence,
\begin{equation}\label{jd:eq:oproj_comm_old}
\OP_{T_X \SM} =\OP_{T_X \M_r}  \OP_{T_X S^{nm-1}} = \OP_{T_X S^{nm-1}} \OP_{T_X \M_r}
\end{equation} 
is an orthogonal projection on the intersection of $T_X \M_r$ and $T_X S^{nm-1}$. Indeed, since 
$$
\VEC(X)\VEC\left(X\right)^\top = (V\otimes U) \VEC(S) \left(\VEC(S)\right)^\top (V^\top \otimes U^\top),
$$
and $$UU^\top ( I -  UU^\top) = 0, \quad VV^\top ( I -  VV^\top) = 0,$$
we get
$$
	\OP_{T_X \M_r} \VEC(X)\left(\VEC(X)\right)^\top = \VEC(X)\left(\VEC(X)\right)^\top =\VEC(X)\left(\VEC(X)\right)^\top \OP_{T_X \M_r}.
$$
Finally, since $\OP_{T_X S^{nm-1}} = I - \VEC(X) \left(\VEC(X)\right)^\top$
\[
\begin{split}
	\OP_{T_X \SM} \OP_{T_X \SM} =&\ \OP_{T_X \M_r}  \OP_{T_X S^{nm-1}} =\ \OP_{T_X \M_r} ( I - \VEC(X)\left(\VEC(X)\right)^\top)\\ 
=&\ \OP_{T_X \M_r}  - \VEC(X)\left(\VEC(X)\right)^\top = \ \OP_{T_X S^{nm-1}}\OP_{T_X \M_r},
\end{split}		
\]
which completes the proof.
\end{proof}

\subsection{Derivation of Jacobi correction equation on $\SM$}
Let us derive the generalization of the original Jacobi correction equation, which is the Riemannian Newton method on $\SM$.
%To do so we need to find one step of the Riemannian Newton method on $\SM$.
Using \eqref{jd:eq:oproj} and notation $x = \VEC (X)$ we obtain
\begin{equation}\label{jd:eq:smgrad}
	\grad \RQ (x) =\OP_{T_X \SM}\nabla\RQ (x) = \OP_{T_X \M_r}(I - xx^\top) \nabla \RQ (x) = 2\OP_{T_X \M_r}(I - xx^\top) Ax.
\end{equation}
Similarly to \eqref{jd:eq:hess_sphere} using \eqref{jd:eq:oproj} we get
\[
\begin{split}
	\Hess_X\, \RQ (x) [\xi] =& 2 \OP_{T_X \SM} (A \xi + \dot{\OP}_{T_X \SM} A x) = \\
		& 2 \OP_{T_X \SM}(A \xi - x\, \xi^\top A x - \xi x^\top A x + \dot{\OP}_{T_X \M_r} Ax),
\\ & \xi \in T_{X} \SM.
\end{split}
\]
According to \eqref{jd:eq:oproj} $\OP_{T_X \SM} x = \OP_{T_X \M_r} \OP_{T_X S^{nm-1}} x = 0$, thus
\[
\begin{split}
	\Hess_X\, \RQ (x) [\xi] =\ &2\OP_{T_X \SM} (A - (x^\top A x) I) \xi + \OP_{T_X \SM} \dot{\OP}_{T_X \M_r} Ax = \\
	&2\OP_{T_X \M_r} (I - x x^\top) (A - (x^\top A x) I) \xi + \OP_{T_X \SM} \dot{\OP}_{T_X \M_r} Ax,
\end{split}
\]
where the part $\OP_{T_X \SM} \dot{\OP}_{T_X \M_r} Ax$ corresponds to the curvature of the low-rank manifold.
This term contains inverses of singular values. Singular values can be small if the rank is overestimated. This, therefore, leads to difficulties in numerical implementation.
Similarly to \cite{ksv-manprec-2016} we omit this part and obtain an inexact Newton method, which can be viewed as a constrained Gauss-Newton method.
Omitting $\OP_{T_X \SM} \dot{\OP}_{T_X \M_r} Ax$ we get
$$
	\Hess_X\, \RQ (x) [\xi] \approx 2\OP_{T_X \M_r} (I - x x^\top) (A - \RQ (x) I) \xi,
$$
or in the symmetric form 
\begin{equation}\label{jd:eq:smhess}
	\Hess_X\, \RQ (x) [\xi] \approx 2\OP_{T_X \M_r} (I - x x^\top) (A - \RQ (x) I) (I - x x^\top)\OP_{T_X \M_r}  \xi.
\end{equation}
Using \eqref{jd:eq:smgrad} and \eqref{jd:eq:smhess} we can write the linear system arising in the inexact Newton method as
\begin{equation}\label{jd:eq:newton_our}
\begin{split}
(I - x x^\top) \left[ \OP_{T_X \M_r}(A - \RQ (x) I)\OP_{T_X \M_r}\right] &(I - x x^\top) \xi = - \OP_{T_X \M_r}(I - x x^\top) Ax, \\ \xi^\top x = 0, &\quad \xi \in {T_X \M_r} .
	%\Hess_X\, \RQ (x) [\xi] =  - \grad \RQ (x), \quad \xi \in T_{X} \SM,
\end{split}
\end{equation}
which has the form similar to the original Jacobi correction equation \eqref{jd:eq:jac_cor} with $(A - \RQ (x) I)$ projected on~$T_X \M_r$.

Equation \eqref{jd:eq:newton_our} is a linear system of size $nm\times nm$, but the number of unknown elements is equal to dimension of the tangent space $(n+m) r - r^2 - 1$.
Hence, the next step is to derive a \emph{local linear system} that is of smaller size and is useful for the numerical implementation.
The following proposition holds.
\begin{prop}
The solution of \eqref{jd:eq:newton_our} written as $$ \xi = \VEC (U_\xi V^\top + U V^\top_\xi + U S_\xi V^\top), $$ %where $U_\xi, V_\xi$ and $S_\xi$ are 
can be found from the local system
\begin{equation}\label{jd:eq:lrj}
	(I - BB^\top) (A- \RQ (x) I)_\mathrm{loc} (I - BB^\top) \tau_\xi = - (I - BB^\top) g, \quad B^\top\tau_\xi = 0,
\end{equation}
where\footnote{For an $nm\times nm$ matrix $C$ we introduced notation 
\[
\begin{split}
&C_{v,v} = (V_k^\top  \otimes I_n) C( V_k \otimes I_n) \in \F^{nr \times nr},\\
&C_{v,u} = (V_k^\top  \otimes I_n) C  (I_m\otimes U_k) \in \F^{nr \times mr},\\
&C_{v,vu}  = (V_k^\top  \otimes I_n) C ( V_k \otimes U_k) \in \F^{nr \times r^2}.
\end{split}
\]
Matrices $C_{u,v},C_{u,u}, C_{u,vu} $ and $C_{vu,v},C_{vu,u},C_{vu,vu}$ are defined likewise.}
$$
	\tau_\xi = \begin{bmatrix} \VEC (U_\xi) \\ \VEC (V^\top_\xi) \\ \VEC (S_\xi) \end{bmatrix}, \quad
	g = \begin{bmatrix} A_{v,v}\,\VEC (U S) \\ A_{u,u}\,\VEC (S V^\top) \\ A_{vu,vu}\,\VEC (S) \end{bmatrix},
\quad 
B = 
	\begin{bmatrix}
			I_r \otimes U & 0 & 0 \\
	0 & V \otimes I_r & 0 \\
	 0 & 0 & \VEC (S) \\
	\end{bmatrix},
$$
$$
 (A -  \RQ (x) I)_\mathrm{loc} = 
\begin{bmatrix}
	(A - \RQ (x) I)_{v,v}  & (A - \RQ (x) I)_{v,u} & (A - \RQ (x) I)_{v,uv} \\
	(A - \RQ (x) I)_{u,v} & (A - \RQ (x) I)_{u,u} & (A - \RQ (x) I)_{u,vu} \\
	(A - \RQ (x) I)_{vu,v} & (A - \RQ (x) I)_{vu,u} & (A - \RQ (x) I)_{vu,vu}
\end{bmatrix},
$$

\end{prop}
\begin{proof}
Notice that $\OP_{T_X \M_r}$ is a sum of three orthogonal projections 
\[
\begin{split}
\OP_{T_X \M_r} &= \OP_1 + \OP_2 + \OP_3, \\
\OP_1 =  VV^\top \otimes (I_n - UU^\top),\quad   \OP_2 &= (I_m - VV^\top) \otimes UU^\top, \quad \OP_3 = VV^\top \otimes UU^\top
\end{split}
\]
Since $\OP_i \OP_j = \mathrm{O}$,  $i\not=j$ and $\OP_i^2 = \OP_i$ we obtain
\begin{equation}\label{jd:eq:jacobi_prelim}
	\begin{bmatrix}  \OP_1 \\ \OP_2 \\ \OP_3 \end{bmatrix}
	(I - x x^\top) (A - \RQ (x) I) (I - x x^\top)
\begin{bmatrix}  \OP_1 & \OP_2 & \OP_3 \end{bmatrix} 
\begin{bmatrix}  \OP_1\xi \\ \OP_2\xi \\ \OP_3\xi \end{bmatrix} = 
\begin{bmatrix}  \OP_1 \\ \OP_2 \\ \OP_3 \end{bmatrix}(I - x x^\top) Ax.
\end{equation}
It is easy to verify that 
\[
\begin{split}
	\OP_1 (I - xx^\top) = &\ \OP_1 = (V\otimes I_n) (V^\top \otimes (I_n - UU^\top)), \\
	  \OP_2 (I - xx^\top) = &\  \OP_2 = (I_m \otimes U) ((I_m - VV^\top) \otimes U^\top), \\
	 \OP_3  (I - xx^\top) = &\ (VV^\top \otimes UU^\top) (I - (V\otimes U) \VEC(S) \left(\VEC(S)\right)^\top (V^\top \otimes U^\top)) =\\
	 & (V\otimes U) (I_{r^2} -  \VEC(S) \left(\VEC(S)\right)^\top ) (V^\top \otimes U^\top).
\end{split}
\]
Then from \eqref{jd:eq:parametrization}
\[
\begin{split}
	&\OP_1 \xi =  V \otimes (I_n - UU^\top)\, \VEC (U_\xi), \\
	&\OP_2 \xi = (I_m - VV^\top) \otimes U\, \VEC (V_\xi^\top), \\
	&\OP_3 \xi =  V \otimes U\, \VEC (S_\xi), \\
\end{split}
\]
Thus, the first block row in \eqref{jd:eq:jacobi_prelim} can be written as
\[
\begin{split}
V\otimes (I_n - UU^\top)( 
&\underbrace{(V^\top \otimes I)(A - \RQ (x) I)(V \otimes I_n)}_{(A - \RQ (x) I)_{v,v}}(I_r \otimes (I_n - UU^\top)) \VEC(U_\xi) + \\
&\underbrace{(V^\top \otimes I)(A - \RQ (x) I)(I_m \otimes U)}_{(A - \RQ (x) I)_{v,u}}((I_m - VV^\top) \otimes I_r)\VEC(V^\top_\xi)+ \\
&\underbrace{(V^\top \otimes I)(A - \RQ (x) I)(V \otimes U))}_{(A - \RQ (x) I)_{v,uv}}(I_{r^2} -  \VEC(S) \left(\VEC(S)\right)^\top ) \VEC(S_\xi) = \\
V\otimes (I_n - UU^\top)& \underbrace{(V^\top \otimes I) A (V \otimes I_n)}_{A_{v,v}} \VEC (US).
\end{split}
\]
Since $V$ has full column rank we obtain exactly the first block row in \eqref{jd:eq:lrj}.
Other block rows can be obtained in a similar way.
\end{proof}

\subsection{Retraction}
Similarly to \eqref{jd:eq:retr_sphere} after we obtained the solution $\xi$ from  \eqref{jd:eq:lrj} we need to map the vector $x+\xi$ from the tangent space back to the manifold. % by using \emph{retraction} \cite{todo}.
	The following proposition gives an explicit representation for the retraction on $\SM$.
\begin{prop}
	Let $R_r$ be a retraction from the tangent bundle $T \M_r$ onto $\M_r$, then 
	\begin{equation}\label{eq:jd:retr}
		R(X, \dot{X}) = \frac{R_r(X, \dot{X})}{\|R_r(X, \dot{X})\|},
	\end{equation}
	is a retraction onto $\SM$.
\end{prop}
\begin{proof}
	To verify that $R$ is a retraction we need to check the following properties \cite{ao-retract-2014}
	\begin{enumerate}
		\item Smoothness on a neighborhood of the zero element in $T\SM$;
		\item $R(X,0) = X$ for all $X\in \SM$;
		\item $\left.\frac{d}{dt}R(X, t\dot{X})\right|_{t=0} = \dot{X}$ for all $X\in \SM$ and $\dot{X}\in T_X\SM$.
	\end{enumerate}
	The first property follows from the smoothness of $R_r$. 
	The second property holds since $R_r(X,0) = X$ and  $\|X\|=1$ for $X\in \SM$.
	Let us verify the third property:
	\begin{equation}\label{jd:eq:retr_proof}
	\begin{split}
		&\left.\frac{d}{dt}R(X, t\dot{X})\right|_{t=0} = 
		\left.\frac{d}{dt}\left(\frac{R_r(X, t\dot{X})}{\|R_r(X, t\dot{X})\|}\right)\right|_{t=0} = \\
		&\frac{\left.\frac{d}{dt}R_r(X, t\dot{X})\right|_{t=0}\left.\|R_r(X, t\dot{X})\right|_{t=0}\| - \frac{d}{dt}\left.\|R_r(X, t\dot{X})\|\right|_{t=0} \left.R_r(X, t\dot{X})\right|_{t=0}}{\|\left.R_r(X, t\dot{X})\right|_{t=0}\|^2}.
	\end{split}
	\end{equation}
	Since $(X, \dot{X}) = 0$ for $X\in \SM$, we get
	\[
	\begin{split}
		\left.\frac{d}{dt}\|R_r(X, t\dot{X})\|\right|_{t=0} & = \\ &\frac{\left.\left(\frac{d}{dt} R_r(X, t\dot{X}), R_r(X, t\dot{X})\right)\right|_{t=0} + \left.\left(R_r(X, t\dot{X}), \frac{d}{dt} R_r(X, t\dot{X})\right)\right|_{t=0}}{2\left.\|R_r(X, t\dot{X})\right|_{t=0}\|} \\
		&= \frac{(\dot{X}, X) + ({X}, \dot{X})}{2\|X\|} = 0, \quad X \in \SM, \quad \dot{X} \in T_X \SM.
	\end{split}
	\]
	Substituting the latter expression into \eqref{jd:eq:retr_proof} and accounting for 
	$$\left.\|R_r(X, t\dot{X})\right|_{t=0}\|= \|R_r(X, 0)\| =\|X\|=1$$ we obtain $\left.\frac{d}{dt}R(X, t\dot{X})\right|_{t=0} = \dot{X}$ which completes the proof.
\end{proof}

\begin{remark}
	Retraction \eqref{eq:jd:retr} is a composition of two retractions: first on the low-rank manifold $\M_r$ and then on the sphere $S^{n-1}$.
	Note that the composition in the reversed order is not a retraction as it does not map to the manifold $\SM$.
	%Indeed, after retraction on $\M_r$ the normalization may be broken.
	%, while normalization of a vector does not change its rank.
	%Indeed, if one first normalizes the vector and then retracts to $\M_r$, then in general normalization breaks.
\end{remark}

	A standard choice of retraction on $\M_r$ is \cite{ao-retract-2014}
	$$
		R_r(x,\xi) \equiv R_r(x+\xi) = \OP_{\M_r} (x + \xi), 
	$$
	where
	$$
		\OP_{\M_r} (x+\xi) \equiv \argmin_{y\in \M_r} \|y - (x+\xi)\|.
	$$
For small enough correction $\xi$ retraction can be calculated using the SVD procedure~\cite{ao-retract-2014} as follows.
First,
\[
\begin{split}
x + \xi = \text{vec}(USV^\top + U_\xi V^\top + U V^\top_\xi + U S_\xi V^\top) = \\
\VEC \left( 
\begin{bmatrix}
	U & U_\xi
\end{bmatrix}
\begin{bmatrix}
	S + S_\xi & I \\ 
	I & O
\end{bmatrix}
\begin{bmatrix}
	V & V_\xi
\end{bmatrix}^\top
 \right)
\end{split}.
\]
Then we calculate $QR$ decompositions 
$$
	Q_U R_U = 
	\begin{bmatrix}
		U & U_\xi
	\end{bmatrix}, \quad
	Q_V R_V = 
	\begin{bmatrix}
		V & V_\xi
	\end{bmatrix}.
$$
and the truncated SVD with truncation rank $r$ of
$$
	R_U
	\begin{bmatrix}
	S + S_\xi & I \\ 
	I & O
\end{bmatrix} R_V^\top,
$$
with $r$ leading singular vectors $U_r\in \F^{2r\times r}$, $V_r\in \F^{2r\times r}$ and the matrix of leading $r$ singular values $S_r\in \F^{r\times r}$.
Thus, the resulting retraction can be written as
$$
	R_r(x + \xi) = (Q_U U_r) S_r (Q_V V_r)^\top.
$$
and from \eqref{eq:jd:retr} the retraction has the form
\begin{equation}\label{jd:eq:retr_svd}
	R(x,\xi) \equiv R(x + \xi) = (Q_U U_r) \frac{S_r}{\|S_r\|}(Q_V V_r)^\top.
\end{equation}

\subsection{Properties of the local system}\label{jd:sec:locsystprop}
Let us mention several important properties of the matrix $(A- \RQ(x) I)_\text{loc}$.
Assume that we are looking for the smallest eigenvalue $\lambda_1$ and $\RQ(x)$ is closer to $\lambda_1$ than to the next eigenvalue $\lambda_2$, i.e. the matrix $(A- \RQ(x) I)$ is nonnegative definite.

 %If matrix $A- \RQ(x) I$ is positive definite and $V$ and $U$ are of full column rank then diagonal blocks $(A- \RQ(x) I)_{u,u}$, $(A- \RQ(x) I)_{v,v}$ and $(A- \RQ(x) I)_{vu,vu}$ are also positive definite, which means that they can be used e.g. for block Jacobi preconditioning. 
%Note that the closer $\RQ(x)$ to the eigenvalue is, the more it difficult to invert diagonal blocks can be.
%We will discuss how this problem is avoided in Section \ref{jd:sec:prec}.
First, the matrix $(A- \RQ(x) I)_\text{loc}$ is singular. Indeed, a nonzero vector 
$$
\begin{bmatrix}
	\VEC(U) \\ -\VEC(V^\top) \\ 0
\end{bmatrix}
$$ 
is in the nullspace of $(A- \RQ(x) I)_\text{loc}$.
This is the result of nonuniqueness of the representation of a tangent vector without gauge conditions.
However,  $(A- \RQ(x) I)_\text{loc}$ is positive definite on the subspace %of vectors $\tau_z$: 
$$ 
B^\top\tau_z=0,\quad \tau_z = 
\begin{bmatrix}
	\VEC(U_z) \\ -\VEC(V_z^\top) \\ \VEC(S_z)
\end{bmatrix}
%[\left(\VEC(U_z)\right)^\top, \left(\VEC(V_z^\top)\right)^\top, \left(\VEC(S_z)\right)^\top]^\top, 
$$ 
where $B$ is defined in \eqref{jd:eq:lrj}. 
Indeed, 
\begin{equation}\label{jd:eq:pos}
\begin{split}
	&\min_{\substack{B^\top\tau_z = 0, \\ \tau_z\not=0}} \left(\tau_z,(A- \RQ(x) I)_\text{loc}\, \tau_z \right) = \\
	&\min_{\substack{B^\top\tau_z = 0, \\ \tau_z\not=0}} 
		(\VEC(U_z V^\top + U V_z^\top + U  S_z V^\top), 
		(A- \RQ(x) I)  \VEC(U_z V^\top + U V_z^\top + U  S_z V^\top))  =\\
	&\min_{\substack{z \in T_x\SM, \\ z\not=0}} \left(z,(A- \RQ(x) I)\, z\right) \geq 
	 \min_{\substack{z \perp x, \\ z\not=0}} \left(z,(A- \RQ(x) I)\, z\right) \geq \lambda_1 + \lambda_2 - 2\RQ(x). 
\end{split}
\end{equation}
The latter inequality follows from \cite[Lemma~3.1]{notay-jdcg-2002}. 
Hence, if $\RQ(x)$ is closer to $\lambda_1$ than to $\lambda_2$, the matrix is positive definite.

Let us show that the condition number of 
$$
	(I - BB^\top)(A- \RQ(x) I)_\text{loc}(I - BB^\top)
$$
does not deteriorate as $\RQ(x)$ converges to the exact eigenvalue.
The condition number is given as
$$
	\kappa = \frac{\max_{\substack{\tau_z: B^\top\tau_z = 0, \\ \tau_z\not=0}}q(\tau_z)}{\min_{\substack{\tau_z: B^\top\tau_z = 0, \\ \tau_z\not=0}} q(\tau_z)}, \quad 
	q(\tau_z) = \frac{(\tau_z, (A- \RQ(x) I)_\text{loc} \tau_z)}{(\tau_z, \tau_z)}.
$$
Similarly to \eqref{jd:eq:pos} one can show that
$$
	\kappa \leq \frac{\max_{\substack{z: z\perp x, \\ z\not=0}} q(z)}{\min_{\substack{z: z\perp x, \\ z\not=0}} q(z)}.
$$
This expression is a bound for the original Jacobi correction equation and according to \cite{notay-jdcg-2002} its condition number does not grow as $\RQ(x)$ approaches the exact eigenvalue~$\lambda_1$.

\section{Subspace acceleration} \label{jd:sec:sub_acc}

%The overall inexact Newton method looks as follows
Since the considered Newton method is inexact or linear systems are solved approximately, we can additionally do the line search
\begin{equation}\label{jd:eq:xnew}
x_\text{new} = R(x + \alpha_\mathrm{opt} \xi),
\end{equation}
where 
$$
	\alpha_\mathrm{opt} = \argmin_{\alpha} \RQ(R(x + \alpha \xi)),
$$
which can be found from the Armijo backtracking rule \cite{absil-manopt-2009} or simply approximated without retraction as
\begin{equation}\label{jd:eq:ls_qopt}
	\alpha_\mathrm{opt} \approx \argmin_{\alpha} \RQ(x + \alpha \xi),
\end{equation}
which can be solved exactly.

To accelerate the convergence one can utilize vectors obtained on previous iterations in the Jacobi-Davidson manner.
However, to avoid instability and reduce the computational cost we use the \emph{vector transport} \cite{absil-manopt-2009}.
At each iteration we project the basis obtained from previous iterations on the tangent space of the current approximation to the solution.
%As we will observe from numerical experiments this approach applied to the original JD method yields similar convergence rate as without projection, however leads to benefits in low-rank arithmetics. %However, it starts stagnating when method approaches accuracy of approximation with a given rank.
Let us consider this approach in more details.

After $k$ iterations we have the basis $\mathcal{V}_{b-1} = [v_1,\dots,v_{b-1}]$, $b\leq k$ and project it on~$T_{X_k}\M_r$:
$$
\mathcal{\widetilde V}_{b-1} = [\OP_{T_{X_k}\M_r} v_1,\dots,\OP_{T_{X_k}\M_r}v_{b-1}].
$$
If needed we can carry out additional orthogonalization of $\mathcal{\widetilde V}_{b-1}$ vectors.
Note that orthogonalization onto the tangent space is an inexpensive operation since linear combinations of any number of vectors from the tangent space can be at most of rank~$2r$.
Given the solution $\xi_k$ of \eqref{jd:eq:newton_our} next step is to expand $\mathcal{\widetilde V}_{b-1}$ with $v_{b}$ obtained from the orthogonalization of $\xi_k$ with respect to $\mathcal{\widetilde V}_{b-1}$:
\begin{equation}\label{jd:eq:rr_noopt}
\mathcal{V}_{b} = [\mathcal{\widetilde V}_{b-1}, v_b]
\end{equation}
A new approximation to $x$ is calculated using the Rayleigh-Ritz procedure.
Namely, we calculate $\mathcal{V}_b^\top A \mathcal{V}_b$ %, where all elements except for last row and column were calculated on previous iterations.
and then find the eigenpair $(\theta, c)$: 
\begin{equation}\label{jd:eq:rr}
	\mathcal{V}_b^\top A \mathcal{V}_b \, c = \theta c,
\end{equation}
corresponding to the desired eigenvalue.
Finally, the Ritz vector $c$ gives us a new approximation to $x$: 
$$x_{k+1} =  \mathcal{V}_b\, c.$$
We emphasize that the columns of $\mathcal{V}_b$ are from $T_{X_k}\M_r$, therefore there is no problem with the rank growth.
%The problem is that rank of linear combination $\mathcal{V}_m c = c_1 v_1  + \dots + c_m v_m$ can be larger than $r$.
%Indeed, if we sum e.g. two matrices $X_1 = U_1 V_1^\top$ with rank $r_1$ and $X_2 = U_2 V_2^\top$ with rank $r_2$ we get representation with at worst rank $r_1 + r_2$:
%$$
% X_1 + X_2 = U_1 V_1^\top + U_2 V_2^\top = 
% 	\begin{bmatrix}
%		U_1 & U_2
%	\end{bmatrix}
%	\begin{bmatrix}
%		V_1 & V_2
%	\end{bmatrix}^\top
%$$
%Reducing rank of $x_{k+1}$ with desired accuracy allows for rank adaptation.
If one wants to maintain fixed rank $r$ it is required to optimize the coefficients $c$:
$$x_{k+1} =  R(\mathcal{V}_b c_\mathrm{opt}), \quad c_\mathrm{opt} =  \argmin_{c_1,\dots,c_b} \RQ(R(\mathcal{V}_b c)).$$
Optimization can be done, e.g. by using the line search over each of $c_i$ sequentially, starting from the initial guess found from \eqref{jd:eq:rr}.
However, to reduce complexity one can optimize only over the coefficient in front of $v_b$, or simply use $c$ instead of $c_\mathrm{opt}$.

%When dealing with the low-rank case we need to maintain low-rank structure of vectors $v_i$ to make our procedure efficient.
%Due to the orthogonality constraint, the rank of vectors $v_i$ can grow when $i$ increases.
%Indeed, if we sum two matrices $X_1 = U_1 V_1^\top$ with rank $r_1$ and $X_2 = U_2 V_2^\top$ with rank $r_2$ we get representation with at worst rank $r_1 + r_2$:
%$$
% X_1 + X_2 = U_1 V_1^\top + U_2 V_2^\top = 
% 	\begin{bmatrix}
%		U_1 & U_2
%	\end{bmatrix}
%	\begin{bmatrix}
%		V_1 & V_2
%	\end{bmatrix}^\top
%$$
%Hence to avoid rank growth we use rounding procedure and reduce the rank of $\mathcal{V}_m c = c_1 v_1  + \dots + c_m v_m$
% with some accuracy, which allows for rank adaptation of the whole procedure.
%Similarly to \eqref{jd:eq:xnew} we additionally normalize the obtained solution.
%	For $m=2$ we get  the line search \eqref{jd:eq:ls_qopt}, which has Riemannian optimization interpretation.
%Alternatively, if it is know that the solution $x$ belongs to the manifold $\mathcal{M}_r$ we can project $\mathcal{V}_m c$ onto the tangent space~$T_{x_k} \mathcal{M}_r$, where $x_k$ is current approximation to the solution and then retract it to the manifold~$T_{x_k} \mathcal{M}_r$.
%Moreover, using restarts when $m$ becomes large can also be useful. 

\section{Connection with Rayleigh quotient iteration}\label{jd:sec:ii}
If the linear system in \eqref{jd:eq:jac_cor} is solved exactly, JD method without the subspace acceleration is known \cite{svh-jd-2000} to be equivalent to the Rayleigh quotient iteration:
\begin{equation}\label{jd:eq:invit_orig}
\begin{split}
	&(A - \RQ(x_k)I)\tilde x = x_k, \\
	&x_{k+1} = \frac{\tilde x}{\|\tilde x\|}.
\end{split}
\end{equation}
%By contrast to the JD method, in \eqref{jd:eq:invit_orig} we solve a system with larger condition number while no orthogonality condition is present. This makes the orthogonalization more difficult if subspace acceleration is used.
Let us find how the method will look like when we solve \eqref{jd:eq:newton_our} exactly.
On the $k$-th iteration equation \eqref{jd:eq:newton_our} looks as
\[
\begin{split}
(I - x_k x_k^\top) \OP_{T_{X_k} \M_r} (A - \RQ (x_k) I)\OP_{T_{X_k} \M_r}\xi_k &=  -\OP_{T_{X_k} \M_r} (I - x_k x_k^\top)Ax_k, \\
\OP_{T_{X_k} \M_r} \xi_k  = \xi_k,\quad  x_k^\top\xi_k&=0.
\end{split}
\]
Therefore,
\[
\begin{split}
\OP_{T_{X_k} \M_r} (A - \RQ (x_k) I)\OP_{T_{X_k} \M_r}&\xi_k - \alpha x_k = -\OP_{T_{X_k} \M_r} (A - \RQ (x_k) I) x_k,
\end{split}
\]
where
$$
\alpha = x_k^\top\, \left[\OP_{T_{X_k} \M_r} (A - \RQ (x_k) I)\OP_{T_{X_k} \M_r}\right]\xi_k. 
$$
Denoting $\tilde x = x_k + \xi_k$, we obtain
\begin{equation}\label{jd:eq:invit}
\begin{split}
 &\left[\OP_{T_{X_k} \M_r} (A - \RQ (x_k) I)\OP_{T_{X_k} \M_r}\right] \tilde x =  x_k, \quad \OP_{T_{X_k} \M_r} \tilde x = \tilde x, \\
 &x_{k+1} = R(\tilde x).
\end{split}
\end{equation}
where the parameter $\alpha$ was omitted thanks to $R(\alpha\tilde x) = R(\tilde x)$.
%For the same reason we do not need condition $x_k^\top\xi_k=0$ since it can be written $x_k^\top \tilde x = 1$, which defines normalization of~$\tilde x$.
Thus, \eqref{jd:eq:invit} represents the extension of the Rayleigh quotient (RQ) iteration \eqref{jd:eq:invit_orig} to the low-rank case and can be interpreted as a Gauss-Newton method.
%The form of \eqref{jd:eq:invit} is not unexpected, it is just projected version of \eqref{jd:eq:invit_orig}.
%Since the derived Jacobi correction equation is basically Gauss-Newton method.
%Unfortunately, superlinear convergence of \eqref{jd:eq:invit_orig} is lostt approximate Hessian.

One can expect that the JD method converges faster than the RQ iteration \eqref{jd:eq:invit} when systems are solved inexactly.
As we have shown in Sec.~\ref{jd:sec:locsystprop} the condition number of local systems in the proposed JD method does not deteriorate when $\RQ(x_k)$ approaches the exact eigenvalue. 
%This is not the case for systems arising in the RQ iteration.
This property positively influences the convergence, as was investigated for the original JD \cite{notay-jdinner-2005}.
We will illustrate it in the numerical experiments in Sec.~\ref{jd:sec:num}.

\section{Complexity}\label{jd:sec:complexity}
Let us discuss how to solve the Jacobi correction equation numerically for the matrix $A$ given as
\begin{equation}\label{jd:eq:operator}
	A = \sum_{\alpha=1}^R F_\alpha \otimes G_\alpha,
\end{equation}
where matrices $F_\alpha$ and $G_\alpha$ are of sizes $n\times n$ and $m\times m$ correspondingly.
In complexity estimates we additionally assume that $F_\alpha$ and $G_\alpha$ can be multiplied by a vector using $\mathcal{O}(n)$ and $\mathcal{O}(m)$ operations respectively, e.g. they are sparse.
%are sparse or more generally can be multiplied by a vector using $\mathcal{O}(n)$ and $\mathcal{O}(m)$ operations respectively.
As an example, $A$ can be the Laplacian-type operator with low-rank potential.

Even if the initial operator $A$ is sparse, the projected local system $A_\text{loc}$ is usually dense.
Fortunately, a fast matrix-vector multiplication by $A_\text{loc}$ can be done.
Let us consider the multiplication by the first block row of $A_\text{loc}$:
\begin{equation}\label{jd:eq:matvec}
\begin{split}
u &= A_{v,v} \text{vec}(U) + A_{v,u} \text{vec}(V^\top) + A_{vu,vu} \text{vec}(S)  \\
&=(V_k^\top  \otimes I_n) A (\text{vec} (UV_k^\top + U_k V^\top + U_k S V_k^\top)  \\
&=(V_k^\top  \otimes I_n) A (\text{vec} (UV_k^\top + U_k (V^\top +S V_k^\top)),
\end{split}
\end{equation}
where we took into account that the vector from the tangent space $UV_k^\top + U_k V^\top + U_k S V_k$ is of rank $2r$ instead of $3r$ as in the case when summing $3$ arbitrary rank-$r$ matrices.
This slightly decreases the cost of matrix-vector multiplication.
%However, been generalized to high dimensions this knowledge can lead to significant speed-up.  
Finally substituting \eqref{jd:eq:operator} into \eqref{jd:eq:matvec}
\[
\begin{split}
	u = 	(V_k^\top  \otimes I_n)\left( \sum_{\alpha=1}^R F_\alpha \otimes G_\alpha \right) \left( ( V_k \otimes I_n)\text{vec}(U) + (I_n\otimes U_k) \text{vec} (V^\top +S V_k^\top) \right) = \\
 \left( \sum_{\alpha=1}^R (V_k^\top  F_\alpha  V_k) \otimes G_\alpha \right)\text{vec}(U) + 
 \left( \sum_{\alpha=1}^R (V_k^\top  F_\alpha) \otimes (G_\alpha U_k) \right)\text{vec}(V^\top +S V_k^\top).
\end{split}
\]
Calculation of an $r\times r$ matrix $V_k^\top  F_\alpha  V_k$ requires $\mathcal{O}(nr^2 + n r)$ operations.
Multiplication of $V_k^\top  F_\alpha  V_k \otimes G_\alpha $ by a vector costs $\mathcal{O}(mr^2  + m r)$. 
Calculation of $V_k^\top  F_\alpha$ and $G_\alpha U_k$ costs $\mathcal{O}(n^2r)$ and $\mathcal{O}(m^2r)$ respectively.
%The cost of multiplication of $(V_k^\top  F_\alpha) \otimes (G_\alpha U_k)$ by a vector is $\mathcal{O}((n+m)r^2)$.
As a result, matrix-vector multiplication costs $\mathcal{O}((n+m)r^2)$ operations.
Given fast matrix-vector multiplication we can solve \eqref{jd:eq:lrj} by the appropriate Krylov type iterative method.
In the next section we discuss how to construct a preconditioner for this system.

In subspace acceleration we project vectors of $V_b$ \eqref{jd:eq:rr_noopt} onto the tangent space.
Projection of each vector costs $\mathcal{O}((m+n)r^2)$.
Thus, assuming that $r\ll n,m$ the complexity of the whole algorithm is $\mathcal{O}((n+m) r(R+r))$.

%Note that if $F_\alpha$ or $G_\alpha$ are sparse or have a certain sturcture then additional reduction in complexity can be achieved..

\section{Block Jacobi preconditioning of the local system}\label{jd:sec:prec}
In the work \cite{svh-jd-2000} the preconditioner of the type $$M_d = ( I - x x^\top) M ( I - x x^\top)$$ was proposed, where $M$ is an approximation to $A - \RQ(x) I$.
If a system with $M$ can be easily solved, then to solve
$$
	M_d\, y = z,
$$
one can use the explicit formula
\begin{equation}\label{jd:eq:4prec}
	y = - \lambda M^{-1} x - M^{-1} z, \quad \lambda = - \frac{x^\top M^{-1} z}{x^\top M^{-1} x}.
\end{equation}
Following this concept we consider a preconditioner of a type
\begin{equation}\label{jd:eq:prec}
	M_d = (I - B B^\top) M_\text{loc} (I - B B^\top) ,
\end{equation}
where $M_\text{loc}$ is an approximation to $(A - \RQ(x) I)_\text{loc}$.
Even if  $M$ is easily inverted, this might not be the case for the projected matrix $M_\text{loc}$.
Hence, we use a block Jacobi preconditioner 
\begin{equation}\label{jd:eq:bdiag}
\begin{split}
	M_d = (I - B B^\top) 
	\begin{bmatrix}
	A_{v,v} - \RQ (x) I  & 0 \\
	0& A_{u,u} - \RQ (x) I  & 0\\
	0 & 0& A_{vu,vu} - \RQ (x) I 
\end{bmatrix}
(I - B B^\top) 
=\\
\begin{bmatrix}
P^\perp_{U} (A_{v,v} - \RQ (x) I) P^\perp_{U} & 0 \\
	0& P^\perp_{V}(A_{u,u}- \RQ (x) I)P^\perp_{V} & 0\\
	0 & 0& P^\perp_{S}(A_{vu,vu}- \RQ (x) I)P^\perp_{S}
\end{bmatrix},\\
\end{split}
\end{equation}
where the projection matrices $P^\perp_{U}$, $P^\perp_{V}$ and $P^\perp_{S}$ are defined as
\[
\begin{split}
&P^\perp_{U} = I_r \otimes (I_n - U U^\top), \\ 
&P^\perp_{V}=(I_n -  V V^\top ) \otimes I_r,  \\ 
&P^\perp_{S} =  I_{r^2}-\text{vec} (S)\left(\VEC(S)\right)^\top.
\end{split}
\]
Let us note that the system with the matrix $A_{vu,vu}- \RQ (x) I$ can be solved directly since it is of small size $r^2 \times r^2$.
Thus, to solve 
$$
P^\perp_{S}(A_{vu,vu}- \RQ (x) I)P^\perp_{S} y = P^\perp_{S} z, \quad  y^\top \VEC(S) = 0, 
$$ 
a direct formula can be used (it follows directly from \eqref{jd:eq:4prec})
\[
	y = (A_{vu,vu}- \RQ (x) I)^{-1} P^\perp_{S}z - \lambda_S (A_{vu,vu}- \RQ (x) I)^{-1} \text{vec}(S),
\]
where
$$
	\lambda_S =  \frac{\left(\VEC(S)\right)^\top(A_{vu,vu}- \RQ (x) I)^{-1} P^\perp_{S}z}{\left(\VEC(S)\right)^\top(A_{vu,vu}- \RQ (x) I)^{-1} \text{vec}(S)}.
$$
Let us derive formulas for solving
$$
P^\perp_{U} (A_{v,v} - \RQ (x) I) P^\perp_{U} y = z, \quad P^\perp_{U} y = y
$$
or equivalently
$$
(I_r \otimes (I_n - U U^\top))\, (A_{v,v} - \RQ (x) I)\, (I_r \otimes (I_n - U U^\top))\, y = z, \quad (I_r \otimes  U^\top) y = 0,
$$
then
$$
(A_{v,v} - \RQ (x) I) y - (I_r \otimes U) \Lambda = z,
$$
where the matrix $\Lambda$ is chosen to satisfy $(I_r \otimes  U^\top) y = 0$.
For a suitable preconditioner $M_{vv}$ which approximates $(A_{v,v} - \RQ (x) I)$ we have
$$
	 y - M_{vv}^{-1} (I_r \otimes U) \Lambda = M_{vv}^{-1} z,
$$
Multiplying the latter equation by $(I_r \otimes U^\top)$ we obtain
$$
	\Lambda = - \left[(I_r \otimes U^\top)M_{vv}^{-1} (I_r \otimes U)\right]^{-1} M_{vv}^{-1} z,
$$
and
$$
	 y = M_{vv}^{-1} (I_r \otimes U) \Lambda + M_{vv}^{-1} z.
$$
Similarly for 
$$
P^\perp_{V} (A_{u,u} - \RQ (x) I) P^\perp_{V} y = z, \quad P^\perp_{V} y = y
$$
we obtain formulas
$$
	y = M_{uu}^{-1} (V \otimes I_r) \Lambda + M_{uu}^{-1} z, \quad \Lambda = - \left[(V^\top \otimes I_r)M_{uu}^{-1} (V \otimes I_r)\right]^{-1} M_{uu}^{-1} z.
$$
Matrices $\left[(V^\top \otimes I_r)M_{uu}^{-1} (V \otimes I_r)\right]$ and $\left[(I_r \otimes U^\top)M_{vv}^{-1} (I_r \otimes U)\right]$ are of size $r\times r$ and can be inverted explicitly.
The main difficulty is to find $M_{uu}^{-1}$ and $M_{vv}^{-1}$.
Their inversion depends on the particular application. 
For instance, if $M = I\otimes F + G \otimes I$, then the inverse can be approximated explicitly as \cite{khor-prec-2009}
\begin{equation}\label{jd:eq:prec_khor}
M^{-1} \approx \sum_{k=1}^K c_k \, e^{-t_k F} \otimes e^{-t_k G},
\end{equation}
which we use later in numerical experiments. 
Alternatively, one can use inner iterations to solve a system with diagonal blocks. 

Note that similar to the original JD method, our method is not a preconditioned eigensolver.
We use the preconditioner only to solve auxiliary linear systems. 
%thus 
%$$
%\tau_\xi = M_\text{loc}^{-1}B\left(B^\topM_\text{loc}^{-1} B\right)^{-1} B^\topM_\text{loc}^{-1}(I - BB^\top) g - 
%M_\text{loc}^{-1}(I - BB^\top) 
%g.

%Even in this case when $M_\text{loc}$ is block-diagonal we are not able to work with it efficiently as solving linear systems with $mr\times mr$ matrix $A_{v,v} - \RQ (x) I $ and $nr\times nr$ matrix $A_{u,u} - \RQ (x) I $ by direct methods can be costly.
%Moreover, these matrices become ill-conditioned as $\RQ (x)$ approaching $\lambda$, so the usage of iterative methods can also be costly.

%On the other hand, similarly to the standard Jacobi-Davidson approach \cite{todo} conditioning of $P^\perp_{U} (A_{v,v} - \RQ (x) I) P^\perp_{U}$ and $P^\perp_{V}(A_{u,u}- \RQ (x) I)P^\perp_{V}$ does not deteriorate as~$\RQ (x)$ converges to~$\lambda$. 
%So, Krylov iterative methods become advantageous. Note that flexible version that allows for variable precondtioning \cite{todo} should be used.
%On top of it 

%On the one hand we could use formula \eqref{jd:eq:onestep} with $M_{loc}$ equals to be block diagonal part of $A_\text{loc}$.
%However even for $n=1000$ and $r=10$ solving equations with $A_{v,v}$ or $A_{u,u}$ by direct methods becomes quite costly.
%Therefore, Krylov subspace methods are used.

\section{Numerical experiments}\label{jd:sec:num}
In numerical experiments we find approximation to the smallest eigenvalue of the convection-diffusion operator
\[
\begin{split}
&\mathcal{A} u \equiv - \frac{\partial^2 u}{\partial x^2}  - \frac{\partial^2 u}{\partial y^2}  + \frac{\partial u}{\partial x} + \frac{\partial u}{\partial y} + V u, \quad (x,y)\in\Omega\\
&\left. u\right|_{\partial\Omega} = 0,
\end{split}
\]
where $\Omega = (-1/2, 1/2)^2$, and potential $V$ is chosen such that solution is of low rank: $V\equiv V(x,y) = e^{-\sqrt{x^2 + y^2}/10}$.
We use a standard second-order finite difference discretization on a $n\times n$ tensor product uniform grid to discretize second derivatives and backward difference to approximate first derivatives.
The potential $V$ on the grid is approximated by the SVD decomposition with relative accuracy $10^{-10}$ and, hence, represented as a diagonal sum-of-product operator.
The discretized operator $A$ is represented in the form \eqref{jd:eq:operator} with sparse matrices $F_\alpha$, $G_\alpha$ and $R=14$.
%As subspace acceleration in all experiments we use \eqref{jd:eq:xnew}.
%This makes comparison of methods on the chosen model more illustrative and fair.
%For example, no version of the ALS method with subspace acceleration was considered.
\paragraph{Low-rank version and original JD}
Let us compare the behaviour of the original JD method and the proposed low-rank version.
Figure~\ref{jd:fig:lr_vs_full} shows the residual plot with respect to the number of outer iterations.
We set the rank $r=5$, grid size $n=150$.
One can observe that the low-rank version stagnates near the accuracy of the best rank $5$ approximation to the exact eigenvector.

\begin{figure}
	\centering{
		\includegraphics[width=120mm]{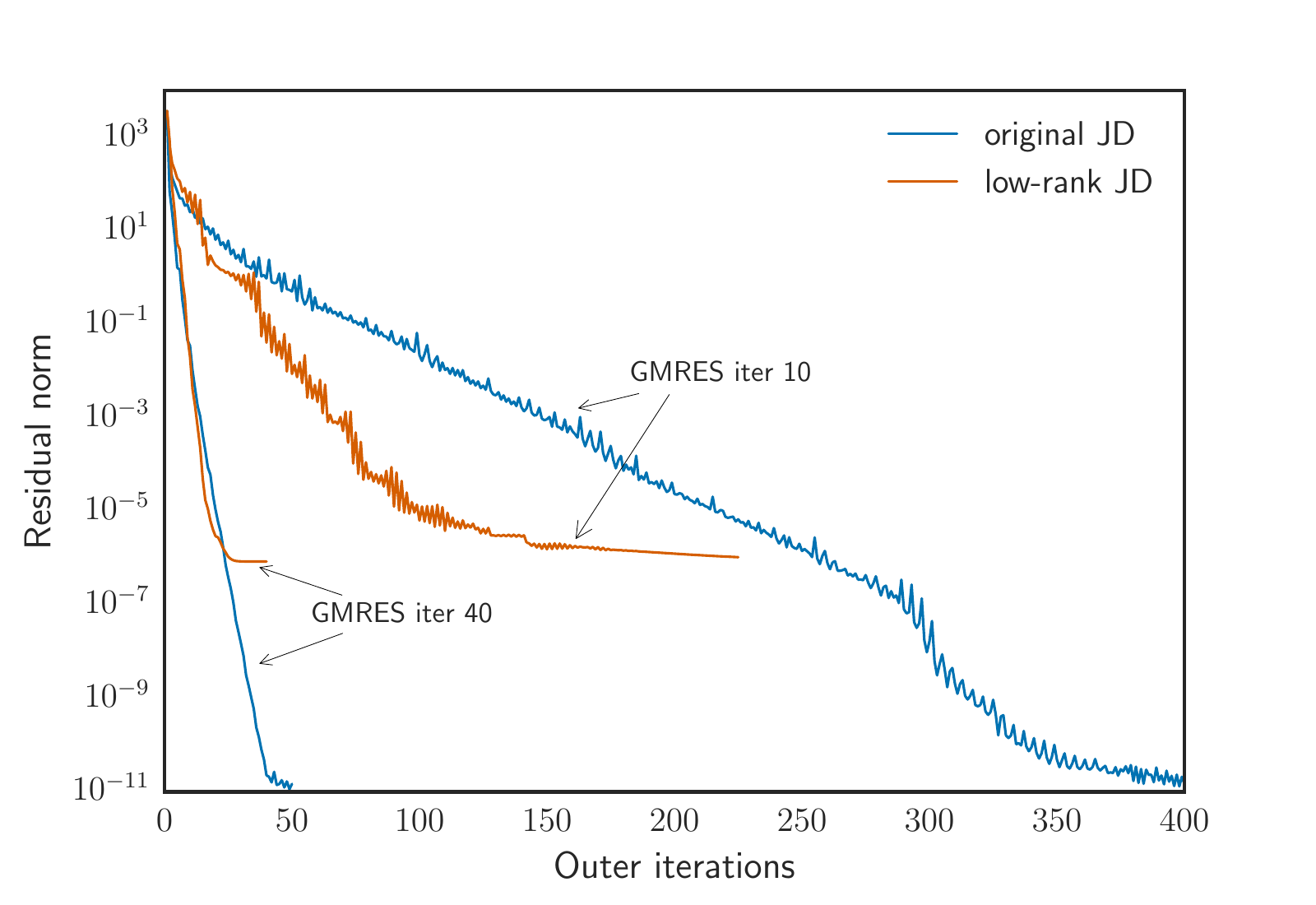}
	}
	\caption{Residual norm w.r.t. the number of outer iterations for the original and the proposed JD methods. Plots are done for different number of inner GMRES iterations to solve linear systems. Parameters: $N=150^2$, $r=5$.}
	\label{jd:fig:lr_vs_full}
\end{figure}

We note that the cost of each inner iteration is different: $\mathcal{O}(nrR)$ for the proposed version and $\mathcal{O}(n^2)$ for the original version, so the proposed version is more efficient for large $n$.
Nevertheless, Figure~\ref{jd:fig:lr_vs_full} shows that our method requires fewer number of less expensive iterations to achieve a given accuracy (before stagnation). The less accurately we solve the system, the more gain we observe.
Such speed-up may happen due to the usage of additional information about the solution, namely that it is of low rank.

\paragraph{Comparison with the low-rank Davidson approach and the Rayleigh quotient iteration}

In this experiment we compare performance of the proposed fixed-rank Jacobi-Davidson approach and the proposed Rayleigh quotient inverse iteration \eqref{jd:eq:invit}.
We also compare them with the ``Davidson'' approach when no projection  $I - x_k x_k^\top$ is done:
\begin{equation}\label{jd:eq:davidson}
\left[\OP_{T_{X_k} \M_r} (A - \RQ (x_k) I)\OP_{T_{X_k} \M_r}\right]\xi_k = - \OP_{T_{X_k} \M_r} r_k, \quad \OP_{T_{X_k} \M_r} \xi_k  = \xi_k.
\end{equation}
%without projection $I - x_k x_k^\top$.
%We will refer to this concept as Davidson approach.
Figure~\ref{jd:fig:jdvsd} illustrates the results of the comparison.
As anticipated, when local systems are solved accurately the Davidson approach stagnates since the exact solution of \eqref{jd:eq:davidson} is~$-x_k$.
So, no additional information is added to the previous approximation $x_k$.
This problem does not occur if local systems are solved inexactly.
For the Rayleigh quotient iteration we observe opposite behaviour due to the deterioration of condition number of local systems.
The Jacobi-Davidson approach yields good convergence in both cases.

\begin{figure}
    \centering
    \begin{subfigure}[b]{0.48\linewidth}        %% or \columnwidth
        \centering
        \includegraphics[width=\linewidth]{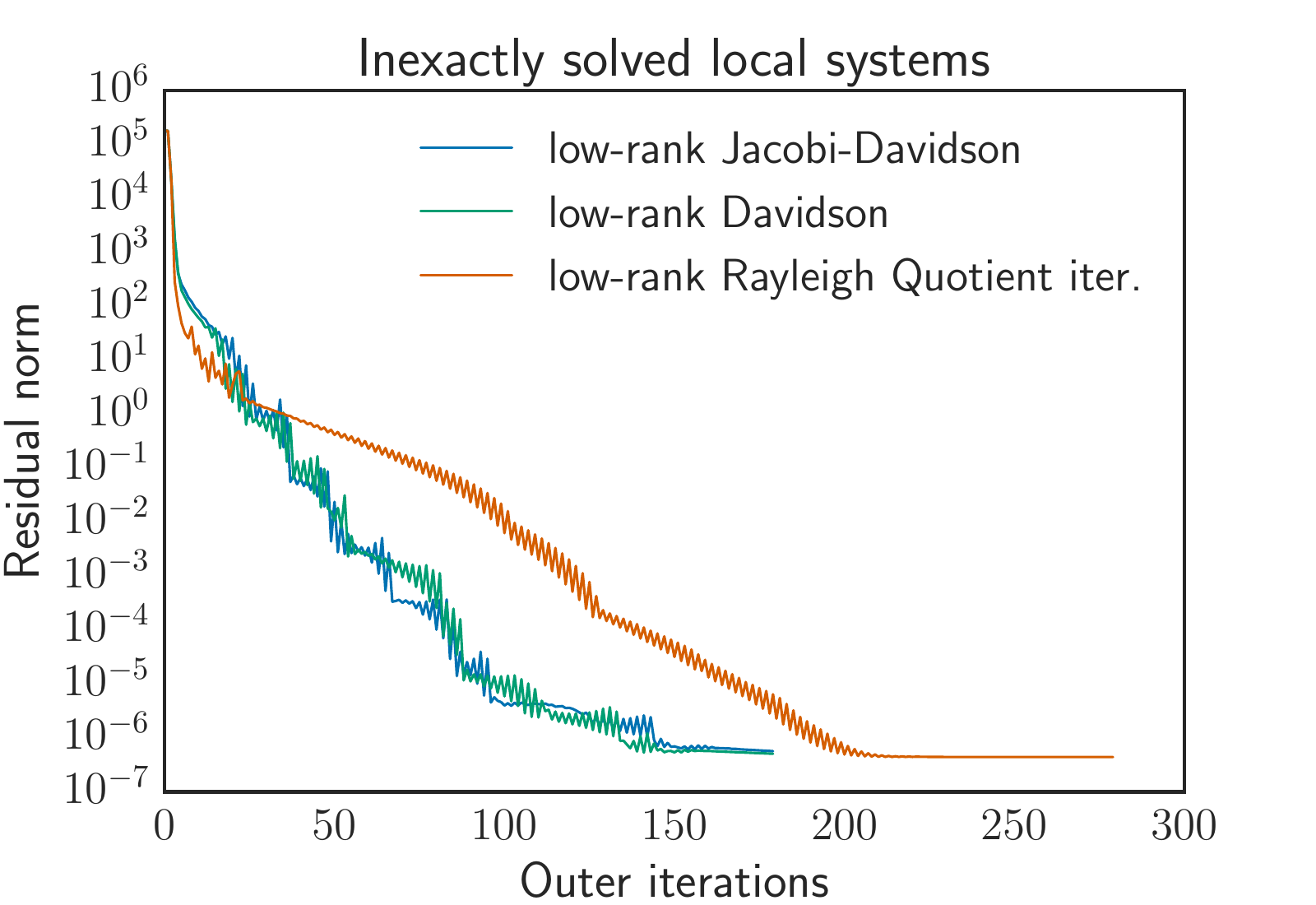}
        \caption{}%150 GMRES iterations}
        \label{jd:fig:dA}
    \end{subfigure}
    \begin{subfigure}[b]{0.48\linewidth}        %% or \columnwidth
        \centering
        \includegraphics[width=\linewidth]{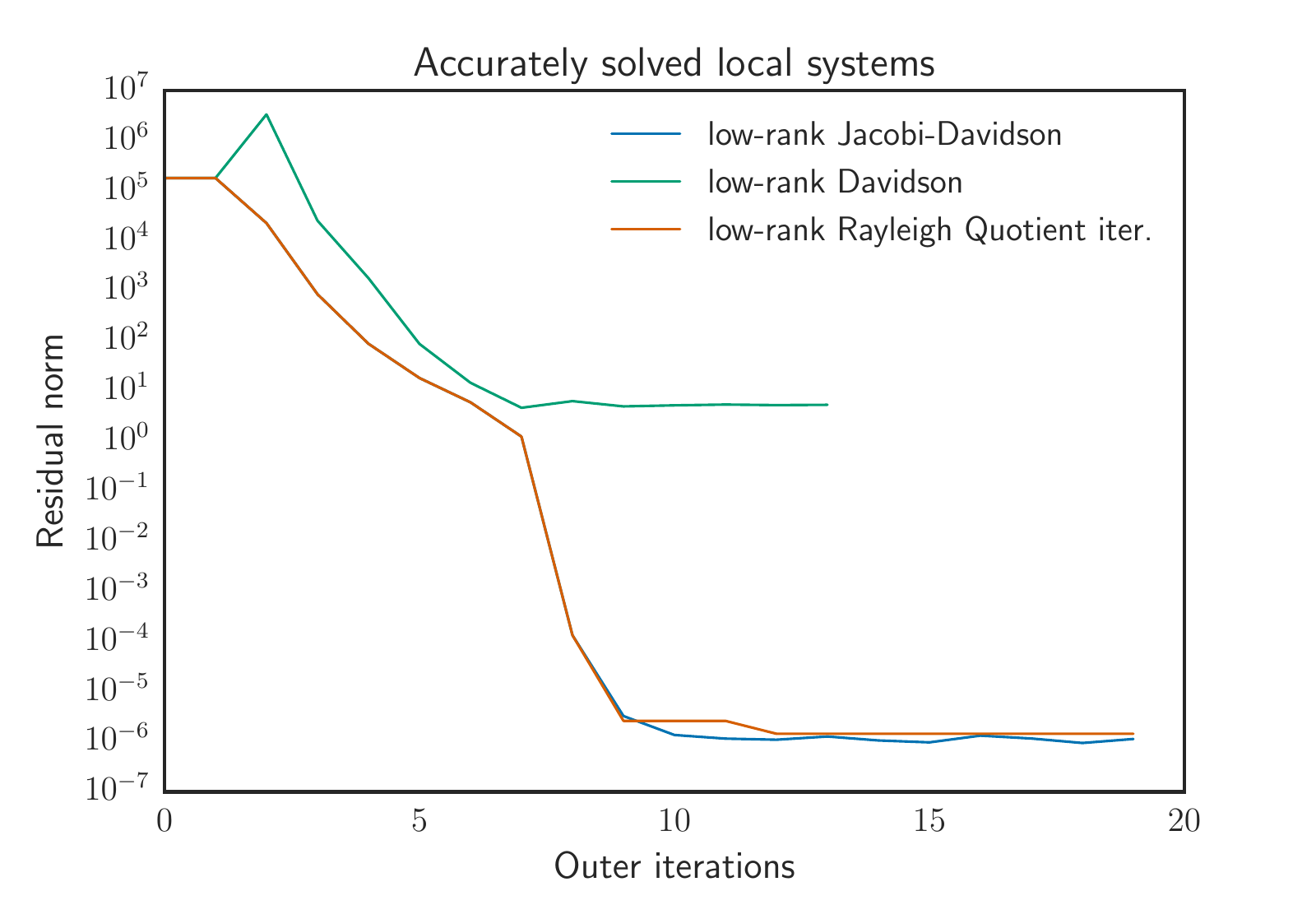}
        \caption{}%600 GMRES iterations}
        \label{jd:fig:dB}
    \end{subfigure}
    \caption{Residual w.r.t. number of outer iterations, $N=2000^2$, $r=3$. Local systems from Figure~\ref{jd:fig:dA} were solved inexactly using $100$ GMRES iterations, while local systems from Figure~\ref{jd:fig:dB} were solved accurately with the preconditioner \eqref{jd:eq:prec_khor}, $K=20$ and $30$ GMRES iterations.}
    \label{jd:fig:jdvsd}
\end{figure}

\paragraph{Comparison with the ALS method}

Alternating linear scheme (ALS) method is the standard approach for low-rank optimization.
The idea is following: given $X = UV^\top$ we minimize Rayleigh quotient $\RQ(x)\equiv \widetilde \RQ (U, V)$ successively over $U$ and $V$.
Minimization over $U$ results in the eigenvalue problem with matrix $A_{v,v}$, while minimization over $V$ results in the eigenvalue problem with matrix $A_{u,u}$.
%Note that $A_{v,v}$ and $A_{u,u}$ are diagonal blocks of the matrix $A_\text{loc}$, so matrix-vector multiplication by $A_{v,v}$ or $A_{u,u}$ is approximately 4 times faster than matrix-vector multiplication by $A_\text{loc}$.

Note that in the proposed JD method we need to solve local systems, while in the ALS approach we solve local eigenvalue problems.
To make comparison fair we ran original JD method to solve local problems in ALS.
We choose the fixed number of iterations as choosing fixed accuracy to solve eigenvalue problems in ALS leads to stagnation of the method.
Since the inner JD solver has two types of iterations: iterations to solve local problem and outer iterations, we need to tune these parameters to get fair comparison.
We tuned them such that each ALS iteration runs approximately the same amount of time as the outer iteration of the proposed JD and gives the best possible convergence.
Results are presented on Figure~\ref{jd:fig:jdvsals}.
On both subfigures the proposed JD method yields the fastest convergence.

\begin{figure}
    \centering
    \begin{subfigure}[b]{0.48\linewidth}        %% or \columnwidth
        \centering
        \includegraphics[width=\linewidth]{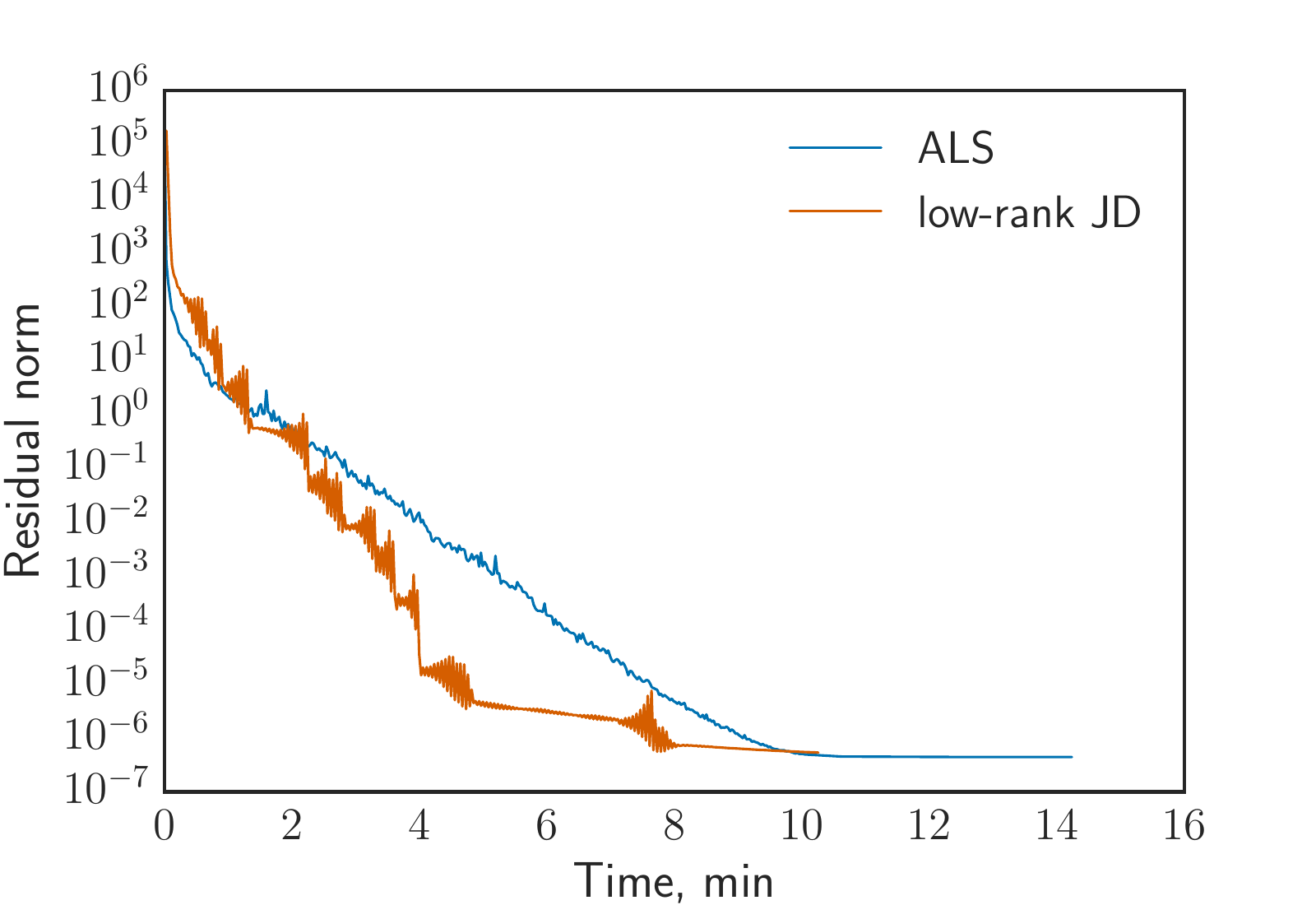}
        \caption{}%150 GMRES iterations}
        \label{jd:fig:alsA}
    \end{subfigure}
    \begin{subfigure}[b]{0.48\linewidth}        %% or \columnwidth
        \centering
        \includegraphics[width=\linewidth]{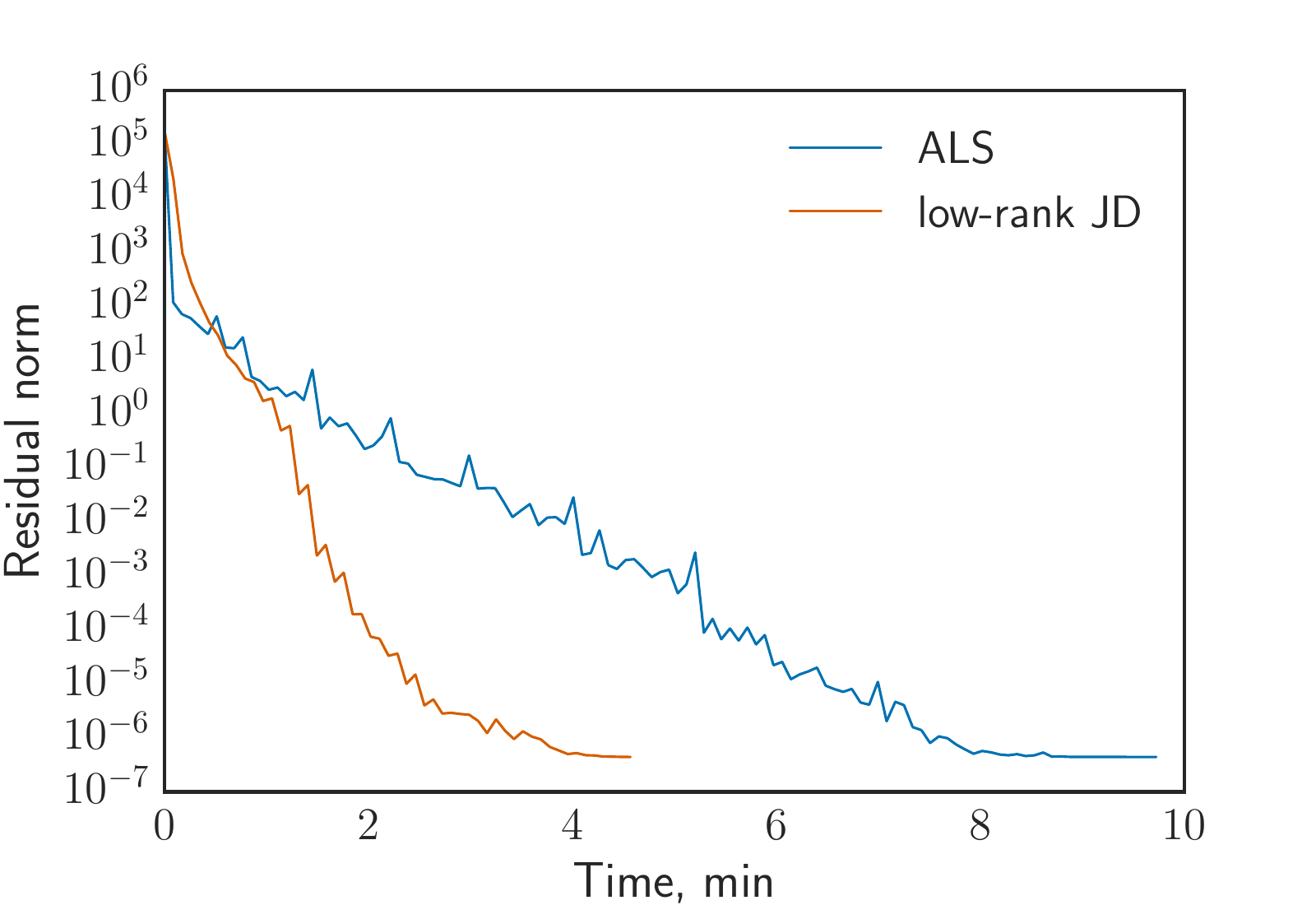}
        \caption{}%600 GMRES iterations}
        \label{jd:fig:alsB}
    \end{subfigure}
    \caption{Residual w.r.t. time for ALS and the proposed JD methods. Figures~\ref{jd:fig:alsA} and \ref{jd:fig:alsB}  correspond to 150 and 600 GMRES iterations to solve local problem of the proposed JD. Parameters of local problems in ALS were chosen to give similar to the proposed JD time of each outer iteration, $N=2000^2$.}
    \label{jd:fig:jdvsals}
\end{figure}

\paragraph{Subspace acceleration}
In this part we investigate the behaviour of the subspace accelerated version proposed in Sec.~\ref{jd:sec:sub_acc}.
First, on Figure~\ref{jd:fig:sub_acc_full} we compare he original subspace acceleration and the version with vector transport when subspace is projected onto the tangent space of the current approximation.
No restarts are used.
As anticipated the projected version stagnates when accuracy of approximation equals error of low-rank approximation.
Apart from that, the convergence behaviour of the methods is comparable, but the projected version is more suitable for  low-rank calculations.
To illustrate this point we provide Figure~\ref{jd:fig:sub_acc_lr}, where the projected version is compared with the version with no projection.
The latter one is implemented with hard rank thresholding of linear combination \eqref{eq:jd:retr}.
No additional optimization over coefficients besides Rayleigh-Ritz procedure is done.
As we observe from the figure, the projected version outperforms the version without projection.
The point is that we exactly optimize coefficients on the tangent space since no rank thresholding in this case is required.
If vectors do not belong to the tangent space, rank rapidly grows with the subspace size and rank thresholding can introduce significant error.

\begin{figure}[t]
  \begin{multicols}{2}
    \hfill
    \includegraphics[width=60mm]{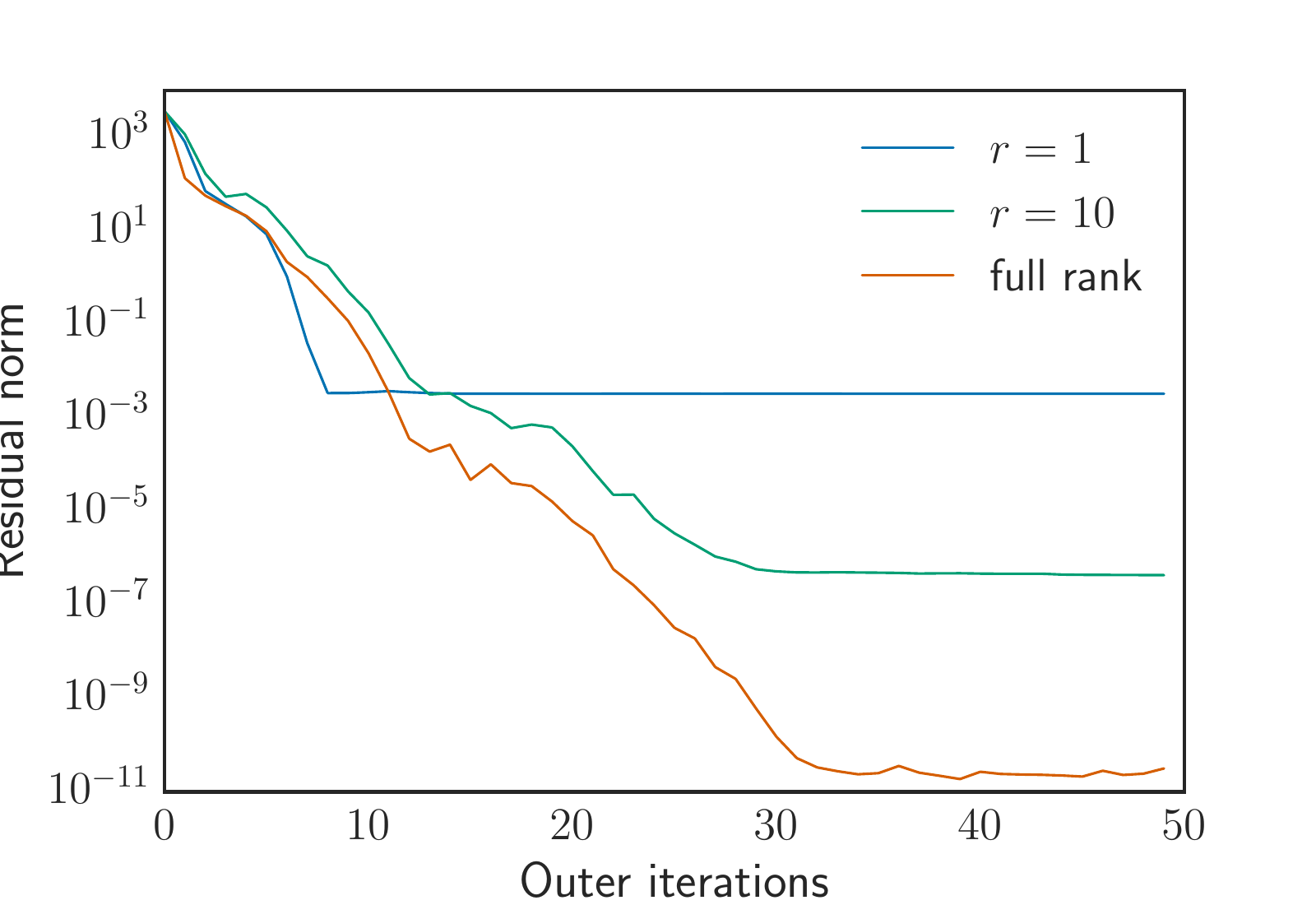}
    \hfill
    \caption{Comparison of original subspace accelerated version of JD and the version with vector transport. In both cases full version of JD with no restarts is used. $N=150^2$, local systems are solved using $150$ GMRES iterations.}
    \label{jd:fig:sub_acc_full}
    \hfill
    \includegraphics[width=60mm]{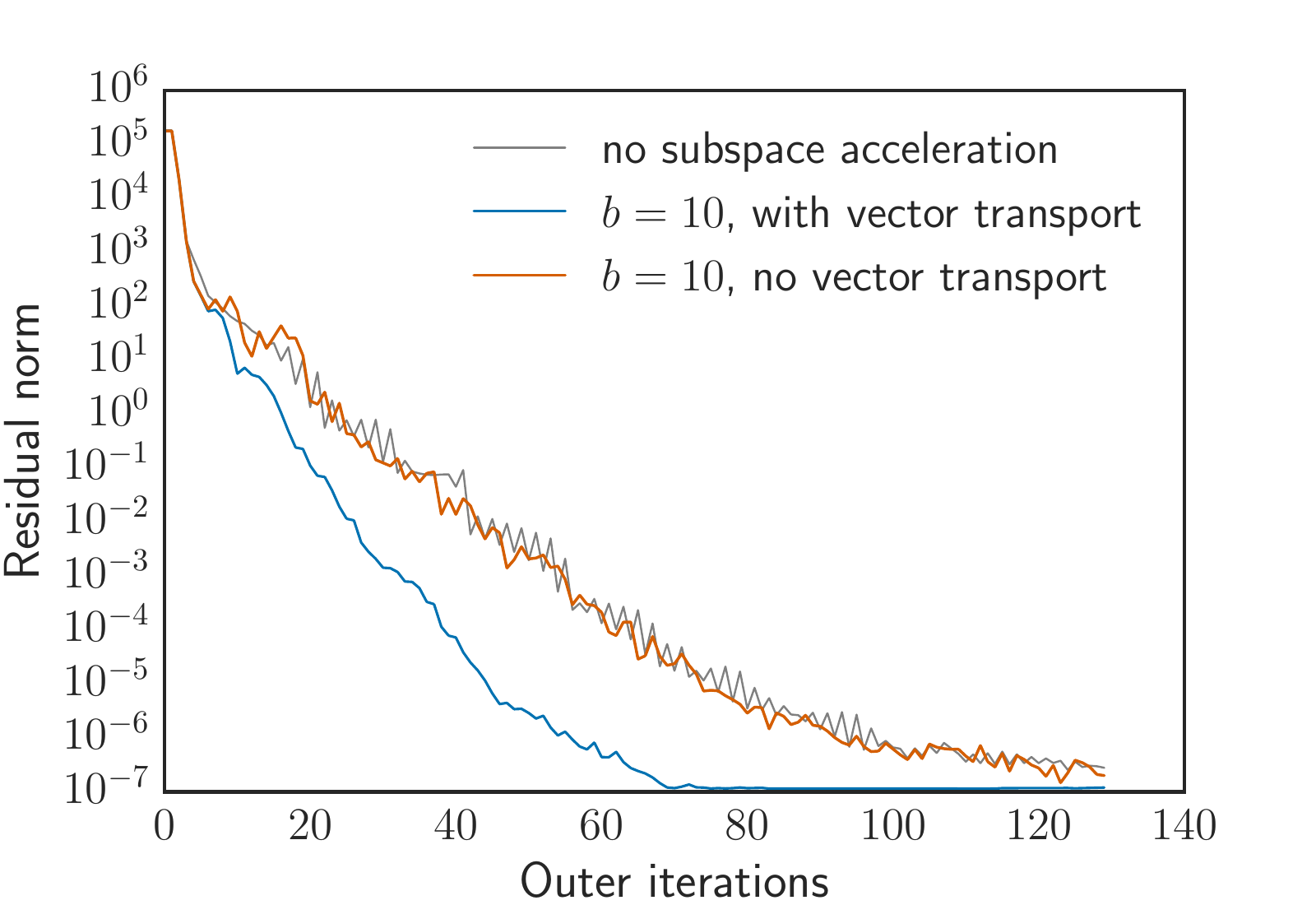}
    \hfill
    \caption{Proposed JD method with subspace acceleration for two cases: when search subspace is projected onto the tangent space (with vector transport) and when no projection is done. Parameters: $N=2000^2$, $r=5$, $150$ GMRES iterations to solve local systems.}
    \label{jd:fig:sub_acc_lr}
  \end{multicols}
\end{figure}

\section{Related work}
%The initial Jacobi-Davidson method was proposed in \cite{todo} and its manifold interpretation was considered in \cite{todo}.
%See a review article \cite{todo} for implementaition details and theoretical properties.
Eigenvalue problems with low-rank constraint are usually considered in literature in the context of more general low-rank decompositions of multidimensional arrays, e.g. the tensor train decomposition \cite{osel-tt-2011}.
Two-dimensional case naturally follows from the multidimensional generalization.

There are two standard ways to solve eigenvalue problems in low-rank format: optimization of Rayleigh quotient based on alternating minimization, which accounts for multilinear structure of the decomposition, and iterative methods with rank truncation.
The first approach has been developed for a long time in the matrix product state community \cite{schollwock-2011,white-dmrg-1992,ostlund-dmrg-1995}.
We also should mention altenating minimization algorithms that were recently proposed in the mathematical community. They are based either on the alternating linear scheme (ALS) procedure \cite{holtz-ALS-DMRG-2012,dkos-eigb-2014} or on basis enrichment using alternating minimal energy method (AMEn)~\cite{kressner-evamen-2014,ds-dmrgamen-2015}.
Rank truncated iterative methods include power method \cite{beylkin-2002,beylkin-high-2005}, inverse iteration \cite{khst-eigen-2012}, locally optimal block preconditioned conjugate gradient method \cite{tobler-htdmrg-2011,lebedeva-tensorcg-2010,lebedeva-tensornd-2011}.
For more information about eigensolvers in low-rank formats see \cite{larskres-survey-2013}.
To our knowledge no generalization of the Jacobi-Davidson method was considered.

%The most similar work is \cite{ksv-manprec-2016}, where authors constructed a preconditioner for linear systems provided that solution has low rank. However, no application to the eigenvalue problems and subspace acceleration were considered.
In \cite{ksv-manprec-2016} authors consider inexact Riemannian Newton method for solving linear systems with a low-rank solution. They also omit the curvature part in the Hessian and utilize specific structure of the operator to construct a preconditioner.

%In \cite{ksv-manprec-2016} authors construct a preconditioner for linear systems with low-rank solution based on the approximate Riemannian Newton method. 
%However, no application to the eigenvalue problem is considered. 
%Moreover, similarly to the original JD method our method is not a preconditioned eigensolver.
%We use the preconditioner only to solve auxiliary linear systems. 
%authots do not consider
In \cite{ro-mp-2016} authors proposed a version of inverse iteration based on the alternating linear scheme ALS procedure, which is similar to \eqref{jd:eq:invit_orig}.
By contrast, the present work considers inverse iteration on the whole tangent space.
We also provide an interpretation of the method as an inexact Newton method.

We note that the proposed approach is considered on the fixed rank manifolds.
Recently desingularization technique was applied to non-smooth variety of bounded-rank matrices $\M_{\leq r}$ \cite{ko-desing-2016pre}.

\section{Conclusions and future work}

The natural next step is to consider generalization to the multidimensional case.
Most of the results can be directly generalized to the tensor train decomposition, e.g.   \eqref{jd:eq:newton_our}, \eqref{eq:jd:retr} and \eqref{jd:eq:invit}.
However, to avoid cumbersome formulas and present the method in the most comprehensible way we restricted the paper to the treatment of the two-dimensional case. Moreover, the correct choice of parametrization of the tangent space and efficient practical implementation worth individual consideration.
We plan to address them in a separate work and test the method on real-world applications.

\section*{Acknowledgements}
The authors would like to thank Valentin Khrulkov for helpful discussions, and Marina Munkhoeva for the careful reading of an early draft of the paper.

%Optimization approach based on alternating minimization can be found in \cite{todo}.
%In the present work we provide comparison with the ALS minimization of Rayleigh quotient.

%In this paper we considered low-rank version of the Jacobi-Davidson method.
%To our knowledge no work  considers the Jacobi-Davidson method in this context. 
%The closest work is \cite{todo}, where authors constructed a preconditioner for linear systems provided that solution has low rank. However, no application to the eigenvalue problem was considered. % solution of linear systems using the Riemannian optimization technique for low-rank tensors.
%They considered preconditioner as a linear system on the tangent space and accounted for explicit representation of Laplacian-like structure.
%Moreover, no discussion of applicability to eigensolvers was provided.

%http://onlinelibrary.wiley.com/store/10.1002/gamm.201490038/asset/368_ftp.pdf;jsessionid=C48C931EAC3D4395008BA80B6E455041.f01t02?v=1&t=is1p8ku6&s=0728f4edd3e33b4dbba086b5bf944505671e29d9

\bibliography{../bibtex/our,../bibtex/dmrg,../bibtex/tensor,../bibtex/iter,jdref}{}
\bibliographystyle{plain}

\end{document}